\documentclass[a4paper,10pt]{amsart}

\usepackage[utf8]{inputenc}

\usepackage{amsmath,amsthm,amssymb,eucal,microtype}

\usepackage{tikz}
\usetikzlibrary{arrows}

\numberwithin{equation}{section}
\newtheorem{thm}[equation]{Theorem}
\newtheorem{lemma}[equation]{Lemma}
\newtheorem{cor}[equation]{Corollary}

\theoremstyle{definition}
\newtheorem{ex}[equation]{Example}

\theoremstyle{remark}
\newtheorem{remark}[equation]{Remark}

\DeclareMathOperator{\ann}{\mathsf{ann}}
\DeclareMathOperator{\supp}{\mathsf{supp}}
\DeclareMathOperator{\supph}{\mathsf{supph}}
\DeclareMathOperator{\Spec}{\mathsf{Spec}}
\DeclareMathOperator{\Proj}{\mathsf{Proj}}
\DeclareMathOperator{\Spc}{\mathsf{Spc}}

\DeclareMathOperator{\im}{\mathsf{im}}
\DeclareMathOperator{\modcat}{\mathsf{mod}}
\DeclareMathOperator{\Modcat}{\mathsf{Mod}}
\DeclareMathOperator{\addcat}{\mathsf{add}}
\DeclareMathOperator{\stmod}{\mathsf{stmod}}
\DeclareMathOperator{\coh}{\mathsf{coh}}

\DeclareMathOperator{\ev}{\mathsf{ev}}
\DeclareMathOperator{\coev}{\mathsf{coev}}

\DeclareMathOperator{\End}{\mathsf{End}}
\DeclareMathOperator{\Hom}{\mathsf{Hom}}
\DeclareMathOperator{\inthom}{\underline{\mathsf{hom}}}
\DeclareMathOperator{\Susp}{\Sigma}
\DeclareMathOperator{\Sq}{\mathsf{Sq}}

\newcommand{\D}{\mathsf D}
\newcommand{\SW}{\mathsf{SW}}
\newcommand{\SHC}{\mathsf{SHC}}
\newcommand{\T}{\mathsf T}
\newcommand{\Sub}{\mathsf S}
\newcommand{\R}{\mathsf R}

\newcommand{\Db}{\D^\mathsf b}
\newcommand{\Dperf}{\D^\mathsf{perf}}
\newcommand{\op}{\mathsf{op}}
\newcommand{\pid}{{\mathfrak p}}

\newcommand{\tensor}{\otimes}
\newcommand{\dtensor}{\otimes^\mathbf{L}}

\newcommand{\unit}{\mathbf1}
\newcommand{\Z}{\mathbb Z}
\newcommand{\RP}{\mathbb{R}\mathsf P}

\newcommand{\F}{\mathsf F}
\newcommand{\HF}{\mathsf H\F}

\newcommand{\equi}{\simeq}
\newcommand{\isom}{\cong}

\newcommand{\homeo}{\approx}

\newcommand{\xto}{\xrightarrow}
\renewcommand{\to}{\longrightarrow}
\newcommand{\into}{\lhook\joinrel\longrightarrow}

\author{Johan Steen}
\address{Institutt for matematiske fag \\ NTNU \\ 7491 Trondheim \\ Norway}
\email{johan.steen@math.ntnu.no}

\author{Greg Stevenson}
\address{Universität Bielefeld \\ Fakultät für Mathematik \\ BIREP Gruppe \\ Postfach 10\,01\,31 \\ 33501 Bielefeld \\ Germany}
\email{gstevens@math.uni-bielefeld.de}

\title{Strong generators in tensor triangulated categories}

\begin{document}

\begin{abstract}
    We show that in an essentially small rigid tensor triangulated category with connected Balmer spectrum there are no proper non-zero thick tensor ideals admitting strong generators.  This proves, for instance, that the category of perfect complexes over a commutative ring without non-trivial idempotents has no proper non-zero thick subcategories that are strongly generated.
\end{abstract}

\maketitle

\section{Introduction}
For some time, particularly since the introduction of compactly generated triangulated categories by Neeman, notions of generation have been central to the study of triangulated categories. More recently, there has been a great deal of interest in notions of generators and dimensions for essentially small triangulated categories. Of particular importance, see for instance \cite{bvdb03} and \cite{rouquier08}, are the triangulated categories admitting strong generators i.e., objects from which the whole category can be built by taking finitely many cones. These objects play an important role in addressing representability questions for (co)homological functors, questions concerning notions of smoothness, and increasingly in studying derived categories of varieties.

Thus it is desirable to ascertain when strong generators exist for (thick subcategories of) triangulated categories and to, if possible, exhibit them. The purpose of this work is to give obstructions to the existence of strong generators in many situations of interest. We prove in Theorem~\ref{thm:connspc} that if $\T$ is an essentially small rigid tensor triangulated category whose spectrum, a topological invariant associated to $\T$, is connected, then no proper non-trivial thick tensor ideal of $\T$ admits a strong generator. For instance, if $R$ is a commutative ring this says that no non-zero proper thick subcategory of the perfect complexes is strongly generated. This can be viewed as a complement to the work of Oppermann and Stovicek \cite{os12} who prove that, when $R$ is noetherian and not necessarily commutative, there are no proper strongly generated thick subcategories in $\Db(\modcat R)$ which contain the perfect complexes.

In fact, the case of commutative rings was the motivation for the abstract result we prove here. It is treated explicitly in Section~\ref{sec:rings} and serves as contrast to the techniques used to prove the theorem. The general result arose from an attempt to both understand how generally such a result could be true and to abstract away the reliance on some ring of operators to provide the obstructions by explicitly constructing objects with arbitrary generation time.

\section{Preliminaries}
In this section we will (very) briefly recall some of the concepts that will be required throughout the paper. Further details can be found in the references given. Let us begin by introducing some notation. Given a triangulated category $\T$ we shall denote the suspension functor by $\Susp$. For an object $x\in \T$ we denote by $\langle x \rangle$ the smallest thick i.e., triangulated and closed under summands, subcategory containing the object $x$. 

\subsection{Generators for triangulated categories}

Let $\T$ be an essentially small triangulated category. An object $g$ of $\T$ is called a \emph{(classical) generator} for $\T$ if
\begin{displaymath}
    \langle g \rangle = \T.
\end{displaymath}
In order to define the notion of a strong generator we need a little preparation. Let $x$ be an object of $\T$. Let $\langle x \rangle_1$ denote the full subcategory of $\T$ consisting of all summands of sums of suspensions of $x$. Put differently $\langle x \rangle_1$ is $\addcat(\Susp^i x \; \vert \; i\in \Z)$, the additive closure of $\{\Susp^i x\; \vert \; i\in \Z\}$. We inductively define $\langle x \rangle_i$ for $i>1$ by
\begin{equation} \label{eq:leveli}
    \langle x \rangle_i = \{b' \in \T \; \vert \; \exists \; \text{triangle}\; a \to b'\oplus b'' \to c \to \Susp a \; \text{with} \; a\in \langle x\rangle_{i-1}, \; c\in \langle x \rangle_1\}.
\end{equation}
Note that $\langle x \rangle_i$ for $i>1$ is automatically closed under sums and suspensions as $\langle x \rangle_1$ is. We say an object $y\in \T$ has \emph{level} $n$ with respect to $x$ if $y\in \langle x \rangle_n$ and $y\notin \langle x \rangle_{n-1}$.

We say that $x$ is a \emph{strong generator} for $\T$ if there is an $n\geq 1$ such that
\begin{displaymath}
    \langle x \rangle_n = \T.
\end{displaymath}
Using the notation we have introduced we can rephrase the statement that $x$ generates (not necessarily strongly) as the equality
\begin{displaymath}
    \bigcup_{i\geq 1} \langle x \rangle_i = \T.
\end{displaymath}
Observe that if $\T$ has a strong generator $x$ then any generator $y$ is strong as $x\in \langle y \rangle_i$ for some $i$.

Further details on generators and strong generators can be found, for instance, in \cite{rouquier08}.

\subsection{Rigid tensor triangulated categories}
Let us now recall the definitions concerning the class of triangulated categories which we will consider.

A \emph{tensor triangulated category} $(\T,\otimes, \unit)$ is a triangulated category $\T$ which is equipped with a symmetric monoidal structure $(\otimes, \unit)$ such that $\otimes$ is an exact functor in each variable. We say a thick subcategory $\Sub$ of $\T$ is a \emph{tensor ideal} if for any $x\in \T$ and $s\in \Sub$ we have $x\tensor s \in \Sub$ i.e., if $\Sub$ is closed under tensoring with arbitrary objects of $\T$.

$\T$ is \emph{closed} if $\otimes$ has a right adjoint which we call the \emph{internal hom} and denote by $\inthom(-,-)$. Suppose $\T$ is closed. For an object $x \in \T$ we set
\begin{displaymath}
    x^\vee = \inthom(x, \unit).
\end{displaymath}
The tensor triangulated category $(\T,\otimes,\unit)$ is \emph{rigid} if for all $x,y\in \T$ the natural morphism
\begin{displaymath}
    x^\vee \otimes y \to \inthom(x,y)
\end{displaymath}
is an isomorphism. In other words, $\T$ is rigid if $x^\vee\otimes-$ is right adjoint to $x\otimes-$.

\subsection{The Balmer spectrum}

Let $(\T,\otimes,\unit)$ be an essentially small tensor triangulated category. Following \cite{balmer05} we associate to $\T$ its spectrum $\Spc \T$. Recall that
\begin{displaymath}
    \Spc \T = \{\mathcal{P}\subsetneq \T \; \vert \; \mathcal{P}\; \text{is prime}\}.
\end{displaymath}
Here $\mathcal{P}$ is \emph{prime} if $\mathcal{P}$ is a proper thick tensor ideal of $\T$ such that whenever $x\otimes y \in \mathcal{P}$, for $x,y \in \T$, we have $x\in \mathcal{P}$ or $y\in \mathcal{P}$. The \emph{Zariski topology} on $\Spc \T$ is given by the basis of closed subsets
\begin{displaymath}
    \big\{ \supp k = \{\mathcal{P}\in \Spc \T \mid k\notin \mathcal{P}\} \mid k\in \T \big\}.
\end{displaymath}
We say a subset $\mathcal{V}$ of $\Spc \T$ is \emph{Thomason} if $\mathcal{V}$ can be written as a union of closed subsets of $\Spc \T$ each of which has quasi-compact complement.

Further details concerning $\Spc \T$ and its role in the classification of thick tensor ideals of $\T$ can be found in \cite{balmer05,balmer10}.

\section{The case of commutative rings}\label{sec:rings}
Let $R$ be a commutative ring. We denote by $\D(\Modcat R)$ the unbounded derived category of $R$.  The \emph{homological support} of a complex $X$ of $R$-modules is defined to be
$$\supph X = \{ \pid\in\Spec R \mid X \dtensor_R k(\pid) \neq 0 \},$$
where $k(\pid)$ is the residue field of the local ring $R_\pid$ and the tensor product is required to be non-zero in the derived category of $R$.  If, in addition, $R$ is noetherian, then a complex $X$ in $\Db(\modcat R)$ satisfies $X \dtensor_R k(\pid) \isom 0$ if and only if $H^*(X)_\pid = 0$ i.e., if and only if $X_\pid \isom 0$.  In this situation $\supph X$ is a closed subset of $\Spec R$, since we have
\begin{displaymath}
    \supph X = \bigcup_{i\in\Z} V\big(\ann H^i(X)\big)
\end{displaymath}
by the previous remark.

Denote by $\Dperf(R) \subseteq \D(\Modcat R)$ the thick subcategory of perfect complexes, and, for a subset $Z\subseteq \Spec R$, by
$$\Dperf_Z(R) = \{ X \in \Dperf(R) \mid \supph X \subseteq Z \}$$
the thick subcategory of perfect complexes supported on $Z$.  A theorem of the aforementioned Thomason \cite{thomason97} (generalizing results of Hopkins and Neeman in the noetherian setting, see \cite{hopkins87,neeman92}) tells us that the lattice of thick subcategories of $\Dperf(R)$ is isomorphic to the lattice of Thomason subsets of $\Spec R$.  The isomorphism is given by sending a thick subcategory $\Sub$ to $\cup_{X\in\Sub} \supph(X)$, and by sending a Thomason subset $Z$ to $\Dperf_Z(R)$.

The aim of this section is to sketch an explicit proof of the following theorem.
\begin{thm} \label{thm:infgen}
    Let $R$ be a noetherian ring with no non-trivial idempotents.  Then no thick subcategory
    $$0 \subsetneq \Sub \subsetneq \Dperf(R)$$
    admits a strong generator.
\end{thm}
As a consequence of Theorem~\ref{thm:connspc} one does not really need the noetherian hypothesis above. In fact one can generalize the proof we give here, by using a passage to the limit argument as in Thomason's classification argument, to remove the noetherian hypothesis. However, this slightly complicates matters so we stick to the noetherian case for simplicity. We should also note that this result, although it does not seem to explicitly appear in the literature, is presumably well known. 

From now on we will assume that $R$ is noetherian. The first step to proving the theorem is to note that if a thick subcategory $\Sub$ admits a generator, then it has to correspond to a closed subset under the classification of thick subcategories (and in fact this is sufficient as well as necessary).
\begin{lemma}
    A thick subcategory $\Sub = \Dperf_Z(R)$ admits a generator if and only if $Z$ is closed.
\end{lemma}
\begin{proof}
    Suppose $\Dperf_Z(R)$ has a generator i.e., $\Dperf_Z(R) = \langle X \rangle$.  Given $Y\in \langle X \rangle$ we necessarily have $\supph Y \subseteq \supph X$, so
    \begin{displaymath}
        \Dperf_Z(R) \subseteq \Dperf_{\supph X}(R).
    \end{displaymath}
    On the other hand we know that, by definition, $\supph X \subseteq Z$ and $\supph \Dperf_Z(R) = Z$. Hence $Z = \supph X$ and is closed as claimed.

    For the other direction, suppose $Z = V(I)$ is a closed subset and $\Sub = \Dperf_Z(R)$. The Koszul complex $K(I)$ is a perfect complex with support $Z$ and so lies in $\Sub$. That it is a generator is a consequence of the classification of thick subcategories: $\langle K(I) \rangle$ is necessarily the thick subcategory of objects supported on $\supph K(I) = Z$.
\end{proof}

This means that a thick subcategory $\Sub$ admitting a generator $X$ is necessarily of the form $\Dperf_{V(I)}(R)$ for some ideal $I\subseteq R$, and this ideal can be chosen to be $I = \ann H^*(X)$.

The idea behind the proof of Theorem \ref{thm:infgen} is to exhibit objects that lie on arbitrarily high levels with respect to the generator $X$.  We begin with two purely algebraic statements that we will need for the proof of the theorem.

\begin{lemma} \label{lemma:ann}
    Let $A \stackrel f\to B \stackrel g\to C$ be an exact sequence of $R$-modules.  Then
    $$\ann A\cdot\ann C \subseteq \ann B.$$
\end{lemma}
\begin{proof}
    Let $a$ and $c$ be elements of $\ann A$ and $\ann C$ respectively.  Then multiplication by $a$ and $c$ on $B$ gives $\im ac \subseteq a\cdot\ker g = a\cdot\im f = 0$. 
\end{proof}

\begin{lemma} \label{lemma:nilpotence}
    Let $R$ be a commutative noetherian ring with connected spectrum and $I\subsetneq R$ a proper ideal.  If the descending chain of powers
    $$I \supseteq I^2 \supseteq I^3 \supseteq \cdots$$
    stabilizes, then $I$ is nilpotent.
\end{lemma}
\begin{proof}
    We will prove the following statement: $V(I)$ is both open and closed in $\Spec R$.  Then, since $\Spec R$ is connected and $I$ is proper, so $V(I)$ is non-empty, it follows that $V(I) = \Spec R$. This is equivalent to the statement that $I$ is nilpotent.

    Choose an $N$ such that $I^N = I^m$ for all $m\geq N$. Given $\pid \in V(I)$ we have an equality $I^NR_\pid = II^NR_\pid$ so by the Nakayama lemma $I^NR_\pid = 0$ i.e., $\mathfrak{p} \notin \supp(I^N)$. Hence $V(I) \subseteq \Spec R \setminus \supp(I^N)$.

    On the other hand suppose $\pid \notin \supp(I^N)$ i.e., $\ann I^N \nsubseteq \pid$. So there is an $r\in R\setminus \pid$ with $rI^N = 0$. It follows that $I^N \subseteq \pid$ and hence $\pid \in V(I)$. Thus $V(I) = \Spec R \setminus \supp(I^N)$ is open and it is, of course, closed. As indicated at the start of the proof this shows $I$ is nilpotent as claimed.
\end{proof}

We are now ready to prove the theorem.
\begin{proof}[Proof of Theorem \ref{thm:infgen}]
    Let $\Sub$ be a thick subcategory admitting a generator $X$ as above i.e., $\Sub = \Dperf_{V(I)}(R)$ where $I$ is the annihilator of $H^*(X)$.  If the descending chain of powers of $I$ stabilizes, then $I$ is nilpotent by the above lemma, so $V(I) = \Spec R$, and the associated thick subcategory is $\Dperf(R)$.  So we may assume that this is not the case.  Note that it in order to show that $\Sub$ does not admit a strong generator, it is sufficient to show that $X$ is not strong.  To do so, we consider the total cohomology of $X$ and the action of $R$ on it.
    
    Let $J$ be an arbitrary ideal of $R$.  The action of $R$ on an $R$-module $M$ factors through $R\xto{\mathrm{can}} R/J$ if and only if $\ann M \supseteq J$, so in particular the action of $R$ on $H^*(X)$ factors through $R\xto{\mathrm{can}} R/I$.  Now fix $n>1$, and construct a Koszul complex for $R/I^n$, namely a bounded complex $K$ of finitely generated free modules such that $H^0(K) \isom R/I^n$ and such that each $H^i(K)$ is annihilated by $I^n$.  In particular $K$ lies in $\langle X \rangle$ and the action of $R$ on $H^*(K)$ factors through $R\xto{\mathrm{can}} R/I^n$, but through no quotient by a lower power of $I$.

    Now we invoke the first lemma:  Let $A \to B \to C \to \Sigma A$ be a distinguished triangle from $\Dperf(R)$ such that $\ann H^*(A) \supseteq I^i$ and $\ann H^*(C) \supseteq I^j$, and the annihilators contain no lower powers of $I$.  Taking homology yields a long exact sequence, and using Lemma \ref{lemma:ann} we obtain that $\ann H^*(B) \supseteq I^{i+j}$.

    This implies that, starting from the generator $X$, in order to build $K$ above, we need to take at least $n-1$ cones.  Since $n$ was chosen arbitrarily, this shows that $X$ is not a strong generator.
\end{proof}

\begin{remark}
    If $R$ is of infinite global dimension, then $\Dperf(R)$ admits no strong generator either.
\end{remark}

\begin{remark}
    The hypothesis that $\Spec R$ is connected is necessary. One can obtain counterexamples by taking $R\times R'$ for $R$ and $R'$ regular of finite Krull dimension.
\end{remark}

\section{Tensor triangulated categories}

In the previous section we showed that non-existence of strong generators for thick subcategories of perfect complexes over a ring can be deduced from connectedness of $\Spec R$.  We now generalize this to the setting of tensor triangulated categories.  Comparing with the previous section, the Balmer spectrum plays the role of $\Spec R$.
\begin{thm} \label{thm:connspc}
    Let $(\T,\tensor,\unit)$ be an essentially small rigid tensor triangulated category.  If $\Spc\T$ is connected as a topological space, then no non-zero and proper thick tensor ideals of $\T$ are strongly generated.
\end{thm}
\begin{remark}
    $\T$ is called \emph{monogenic} if the smallest thick subcategory containing the tensor-unit is $\T$ itself i.e., $\T = \langle\unit\rangle$.  In such categories the notion of a thick tensor ideal coincides with that of a thick subcategory.
\end{remark}

\begin{ex}
	In the case that $\T$ is \emph{not} monogenic the statement cannot in general be improved to cover all thick subcategories. Indeed, let $k$ be a field and consider $\Db(\coh \mathbb{P}^1_k)$, the bounded derived category of coherent sheaves on the projective line over $k$. The structure sheaf $\mathcal{O}$ is exceptional i.e., $\End^*(\mathcal{O}) \cong k$, and thus $\langle \mathcal{O} \rangle$ is equivalent to $\Db(\modcat k)$. Clearly the latter category has a strong generator, namely $k$, and so $\langle \mathcal{O} \rangle$ is a strongly generated thick subcategory of $\Db(\coh \mathbb{P}^1_k)$. We note that it is not a tensor ideal: $\mathcal{O}$ is the tensor unit so the smallest tensor ideal containing it is $\Db(\coh \mathbb{P}^1_k)$. 
\end{ex}

The following result is proved by the second author in the appendix of \cite{biks13}, and constitutes one of the two main ingredients of our proof.
\begin{thm}\label{thm:biks}
    Let $\T$ be a tensor triangulated category, and let $\Sub$ and $\R$ be thick tensor ideals of $\T$.  If $\Sub \stackrel{i_*}\to \T \stackrel{j^*}\to \R$ is a semi-orthogonal decomposition of $\T$ i.e., the functors $i_*$ and $j^*$ both admit right adjoints, then the spectrum of $\T$ decomposes as
    $$\Spc\T \homeo \Spc\Sub \sqcup \Spc\R.$$
    If $\T$ is an essentially small rigid tensor triangulated category, it is sufficient to assume that $\Sub$ is a thick tensor ideal.  Moreover, in this case there is a decomposition $\T \equi \Sub\oplus\R$.
\end{thm}

We will prove the theorem by showing that the inclusion of a strongly generated tensor ideal has a right adjoint; this statement gives, in conjunction with Theorem~\ref{thm:biks}, a contradiction to the connectedness of $\Spc \T$. The second main ingredient is the following result due to Rouquier, generalizing a result of Bondal--Van den Bergh \cite{bvdb03}.  It allows us to produce the desired right adjoint.  Before stating the result we need to recall the notion of locally finitely presented cohomological functor.

A cohomological functor $H\colon \Sub^\op \to \Modcat\Z$ is \emph{locally finitely generated} if for each $x\in\Sub$, there is an object $a$ and a natural transformation $\Sub(-,a)\to H$ that is epimorphic when evaluated on $\Susp^i x$ for all $i$.  Moreover, $H$ is \emph{locally finitely presented} if it is locally finitely generated and, for all $b \in \Sub$, the kernel of any natural transformation $\Sub(-,b)\to H$ is locally finitely generated.

\begin{thm}[{\cite[Corollary 4.17]{rouquier08}}]\label{thm:rouq}
    Suppose that $\Sub$ is a strongly generated triangulated category with split idempotents and that $H\colon \Sub^\op \to \Modcat\Z$ is a cohomological functor.  Then $H$ is representable if and only if $H$ is locally finitely presented.
\end{thm}

We now embark upon the proof of Theorem~\ref{thm:connspc}. The first step is to set up a representability argument by showing that certain hom-functors are locally finitely presented.

\begin{lemma}
    Let $\T$ be an essentially small rigid tensor triangulated category, and $\Sub\stackrel{i_*}\into \T$ a thick tensor ideal that admits a (not necessarily strong) generator.  Then for all $t\in\T$, the functor
    $$H_t := \T\big(i_*(-),t\big)\colon \Sub^\op \to \Modcat\Z$$
    is locally finitely presented.
\end{lemma}
\begin{proof}
    The adjunction between $\tensor$ and the internal hom $\inthom$ in $\T$ gives canonical morphisms
    $$\ev_{x,y}\colon \inthom(x,y)\tensor x \to y \qquad \textrm{and} \qquad \coev_{x,y}\colon x \to \inthom(y, x\tensor y),$$
    namely the counit and unit of adjunction.

    By assumption $\Sub$ admits a generator, say $g$. Rouquier's \cite[Lemma 4.6]{rouquier08} tells us that to check $H_t$ is locally finitely presented it is sufficient to verify it is locally finitely presented at $g$.

    To this end, consider the following completion of $\coev$ into a triangle in $\T$,
    $$\unit \xto{\coev_{\unit,g}} \inthom(g,g) \to z \to \Susp \unit.$$
    Tensoring this triangle with $g$ yields
    \begin{equation} \label{eq:splittri}
        g \isom \unit\tensor g \xto{\coev_{\unit,g}\tensor g} \inthom(g,g) \tensor g \to z \tensor g \to \Susp \unit \tensor g \isom \Susp g,
    \end{equation}
    and we claim that this is a split triangle.  Indeed the composition
    $$g \isom \unit\tensor g \xto{\coev_{\unit,g}\tensor g} \inthom(g,g) \tensor g \xto{\ev_{g,g}} g$$
    is the identity morphism, as this is a triangle identity of the tensor-hom adjunction.  Thus (\ref{eq:splittri}) is split.

    Furthermore, the adjunction yields the following commutative diagram of natural transformations
    $$\begin{tikzpicture}[scale=.75]
        \node (1) at (0,2) {$\T\big(-\tensor\,\inthom(g,g),t\big)$};
        \node (2) at (6,2) {$\T(-,t)$};
        \node (3) at (0,0) {$\T\big(-,\inthom(g,g)^\vee\tensor t\big)$};
        \node (4) at (6,0) {$\T(-,t)$.};
    \path[->,font=\scriptsize]
        (1) edge node[auto] {$(1\tensor \coev_{\unit,g})^*$} (2)
        (1) edge node[auto,swap] {$\gamma$} node[auto] {$\wr$} (3)
        (3) edge node[auto] {$(\coev_{\unit,g}^\vee \tensor \,1)_*$} (4)
        (2) edge[thick,double distance=2pt,-] node {} (4);
    \end{tikzpicture}$$
    Evaluating the lower natural transformation at $\Susp^i g$ gives a morphism
    $$\Sub\big(\Susp^i g,\inthom(g,g)^\vee\tensor t\big) \xto{(\coev_{\unit,g}^\vee \tensor \,1)_*} H_t(\Susp^i g),$$
    where the source can be written as a hom in $\Sub$ since $\Sub$ is a tensor ideal and so contains $\inthom(g,g)^\vee\tensor t$. This map is a (split) epimorphism by the commutativity of the above square and the earlier observation that $\coev_{\unit,g} \tensor g$ is a split monomorphism. Thus $H_t$ is locally finitely generated.
    
    Now choose any natural transformation $\beta\colon\Sub(-,b) \to H_t$, with $b\in\Sub$.  This is represented by a morphism $b\to t$ in $\T$, which we complete to a triangle
    $$a \to b \to t \to \Susp a.$$
    By what we have already proved $H_a$ is locally finitely generated. By construction i.e., via the triangle definining $a$, there is an epimorphism $H_a \to \ker\beta$, showing that $\ker\beta$ is locally finitely generated as well.  Thus $H_t$ is locally finitely presented.
\end{proof}

With this lemma at hand the proof of the main theorem is straightforward.

\begin{proof}[Proof of Theorem \ref{thm:connspc}]
    Suppose $\Sub$ admits a strong generator.  As pointed out above, it is sufficient to show that the inclusion $i_*\colon \Sub \to \T$ admits a right adjoint: in this case Theorem~\ref{thm:biks} implies $\Spc \T$ is not connected, contradicting our hypotheses.

    We now claim that $\T$ can be assumed to be idempotent complete without loss of generality.  To see this, let $\T^\sharp$ denote the idempotent completion of $\T$.  It is an essentially small rigid tensor triangulated category provided $\T$ is, and the fully faithful inclusion $\T \into \T^\sharp$ is an exact monoidal functor inducing a homeomorphism $\Spc\T^\sharp \homeo \Spc\T$ by \cite[Corollary~3.14]{balmer05}.  Moreover, $\Sub^\sharp$ is a strongly generated thick tensor ideal of $\T^\sharp$, and hence it is sufficient to produce a right adjoint to the inclusion of $\Sub^\sharp$ in $\T^\sharp$.

    As $\Sub$ admits a generator the previous lemma tells us that $H_t := \T\big(i_*(-),t\big)$ is locally finitely presented for all $t\in \T$. Since we have assumed that the generator of $\Sub$ is moreover strong we can apply Theorem~\ref{thm:rouq} to deduce that $H_t$ is in fact representable i.e.,
	\begin{displaymath}
        H_t = \T\big(i_*(-),t\big) \cong \Sub(-,i^!t)
	\end{displaymath}
	for some $i^!t$ in $\Sub$. As everything in sight is natural the assignment $t\mapsto i^!t$ gives a functor $i^!$ which is manifestly right adjoint to $i_*$. As explained at the beginning of the proof this yields a contradiction, implying $\Sub$ cannot be strongly generated.
\end{proof}

\begin{cor} \label{cor:quotient}
    Let $\Sub$ be a thick subcategory of $\T$ such that the Verdier quotient $\T/\Sub$ is a monogenic essentially small tensor triangulated category with connected spectrum.  Then any strongly generated thick subcategory $\R$ such that $\Sub \subseteq \R \subseteq \T$ is either $\Sub$ or $\T$.
\end{cor}
\begin{proof}
    Since $\T/\Sub$ is monogenic, the image of $\R$ in $\T/\Sub$ is a thick tensor ideal, and thus by the theorem is strongly generated only if it is $0$ or $\T/\Sub$.  Since strong generation of the image is implied by strong generation of $\R$ itself, it follows that if $\R$ is strongly generated then it coincides with either $\Sub$ or $\T$.
\end{proof}

\section{Applications}

We now give some situations in which one can apply Theorem~\ref{thm:connspc} to see that strongly generated thick subcategories are not so easy to come by.

\subsection{Perfect complexes over commutative rings}
We now revisit the case where $\Dperf(R)$ is the category of perfect complexes over a commutative ring $R$.  In contrast to Section \ref{sec:rings} we no longer require $R$ to be noetherian.  $\Dperf(R)$ is a tensor triangulated category with the derived tensor product $\tensor^{\mathbf L}_R$ and tensor unit $R$.  The monoidal structure is closed via the usual tensor-hom adjunction with $\mathbf R\Hom_R(-,-)$, and makes $\Dperf(R)$ an essentially small rigid tensor triangulated category.  Indeed Balmer shows in \cite[Example~4.4]{balmer10} that $\Spc\big(\Dperf(R)\big) \homeo \Spec R$, so since $\Dperf(R)$ is monogenic Theorem \ref{thm:connspc} generalizes Theorem \ref{thm:infgen}.

\subsection{Stable categories of finite groups}
Let $k$ be a field of characteristic $p$ and $G$ be a finite group whose order is divisble by $p$.  Denote by $\stmod kG$ the stable module category.  The tensor product over $k$ with the diagonal group action makes this into an essentially small rigid tensor triangulated category with tensor unit $k$.  By \cite[Corollary~5.10]{balmer05} there is a homeomorphism
$$\Spc(\stmod kG) \homeo \Proj H^\bullet(G;k).$$
That this is a connected space follows from the indecomposability of $k$ in $\stmod kG$, for instance by the main theorem of \cite{balmer07}.

It follows that $\stmod kG$ has no proper and non-zero thick tensor ideals admitting a strong generator.

Note that whenever $G$ is a $p$-group $\stmod kG$ is monogenic.  Since 
$$\stmod kG \equi \Db(\modcat kG)/\Dperf(kG)$$
we can apply Corollary \ref{cor:quotient} to recover a special case of the aforementioned result of Oppermann and Stovicek: There are no proper strongly generated thick subcategories in $\Db(\modcat kG)$ which contain $\Dperf(kG)$.

\subsection{The finite stable homotopy category}
The stable homotopy category of spectra $(\SHC,\wedge,S^0)$ is a tensor triangulated category with monoidal product the smash product and the sphere spectrum as the unit.  The thick subcategory of compact objects $\SHC^{\mathsf c}$ consists of the finite spectra i.e., suspension spectra of finite CW complexes.  We shall adopt the convention of calling this subcategory the \emph{Spanier--Whitehead category}, namely
$$\SW = \langle S^0 \rangle = \SHC^{\mathsf c} \subseteq \SHC.$$

The category $\SW$ inherits the structure of $\SHC$ making it tensor triangulated.  It is essentially small and moreover rigid; the dual objects $x^\vee$ are the Spanier--Whitehead duals.  Finally, by using the Hopkins--Smith \cite{hs98} classification of thick subcategories, Balmer shows in \cite[Corollary~9.5]{balmer10} that $\Spc(\SW)$ is a connected topological space, and so Theorem \ref{thm:connspc} applies.  It turns out that we can do slightly better: $\SW$ does not admit a strong generator either, as we now proceed to show.

Denote by $\F_2$ the field with two elements, and by $\HF_2$ the Eilenberg--MacLane spectrum which is the representing object of mod $2$ singular cohomology on $\SHC$.  The cup product of singular cohomology is not a suitable invariant when working stably, but instead one can consider the module structure over the mod $2$ Steenrod algebra.  We recall here that this is the $\F_2$-algebra generated by the Steenrod squares $\Sq^i$, which are degree $i$ stable cohomology operations.  While we will not mention the defining properties of the Steenrod squares and rather refer the reader to \cite{margolis83} for a thorough treatment, we recall the fact that if $\alpha \in \widetilde H^1(\RP^{2^k};\F_2) \isom \F_2$ is the non-zero degree one cohomology class of the real projective space, then also
$$\Sq^{2^{k-1}}\cdots\Sq^2\Sq^1(\alpha) \in \widetilde H^{2^k}(\RP^{2^k};\F_2) \isom \F_2$$
is non-zero for any $k>0$.

As the Steenrod squares correspond to morphisms $\Sq^i\colon \HF_2 \to \Susp^i \HF_2$ in $\SHC$, we can express the cohomology class as the following non-zero composition in $\SHC$
$$\RP^{2^k} \stackrel\alpha\to \Susp^1 \HF_2 \xto{\Susp^1\Sq^1} \Susp^2 \HF_2 \to \cdots \to \Susp^{2^{k-1}} \HF_2 \xto{\Susp^{2^{k-1}} \Sq^{2^{k-1}}} \Susp^{2^k} \HF_2.$$
Note that each morphism in this composition vanishes when precomposed with morphisms out of any suspension of the sphere spectrum.  The ``ghost lemma'' for cocomplete triangulated categories, see \cite{christensen98}, in particular Section~7.1, then implies that
\begin{equation} \label{eq:RP}
    \RP^{2^k} \not\in \langle S^0 \rangle^\oplus_k \subseteq \SHC.
\end{equation}
Here $\langle-\rangle^\oplus_k$ is defined as in (\ref{eq:leveli}), but where one allows arbitrary direct sums.

We immediately obtain
\begin{thm}
    No non-zero thick subcategories of the Spanier--Whitehead category $\SW$ admit strong generators.
\end{thm}
\begin{proof}
    From our comments preceding this theorem we see that $\SW$ satisfies the assumptions of Theorem \ref{thm:connspc}, and since the Balmer spectrum is connected, we obtain the result for any \emph{proper} non-zero thick tensor ideal, and thus for all proper non-zero thick subcategories since $\SW$ is monogenic.

    It remains to show that $\SW$ does not admit a strong generator.  As remarked earlier, it is sufficient to show that $S^0$ does not strongly generate.  By Proposition~2.2.4 of \cite{bvdb03} one has
    $$\langle S^0 \rangle_k = \langle S^0 \rangle^\oplus_k \cap \SW,$$
    so by (\ref{eq:RP}) there are objects for any $k$ which are not level $k$ or less with respect to $S^0$.  This shows that $\SW$ is not strongly generated by $S^0$.
\end{proof}

\bibliographystyle{amsplain}
\bibliography{stronggen}

\end{document}